\documentclass[11pt]{article}
\usepackage{amssymb,amsmath,graphicx,amsfonts}
\usepackage{amsmath}
\usepackage{amsfonts}

\newtheorem{theorem}{Theorem}[section]

\newtheorem{lemma}[theorem]{Lemma}
\newtheorem{proposition}[theorem]{Proposition}
 
\newtheorem{remark}[theorem]{Remark}
\newtheorem{example}[theorem]{Example}
\newtheorem{problem}{Problem}
\newtheorem{question}[theorem]{Question}

\newtheorem{definition}[theorem]{Definition}

\newtheorem{precor}{{\bf Corollary}}

\newtheorem{precon}{{\bf Conjecture}}

\newtheorem{predefin}{{\bf Definition}}

\newtheorem{preexm}{{\bf Example}}

\newtheorem{preappl}{{\bf Application}}

\newtheorem{prelem}{{\bf Lemma}}

\newtheorem{preproof}{{\bf Proof.\ }}

\newenvironment{proof}[1]{\begin{preproof}{\rm
               #1}\hfill{$\blacksquare$}}{\end{preproof}}
\newtheorem{presproof}{{\bf Sketch of Proof.\ }}

\newtheorem{prethm}{{\bf Theorem}}

\newtheorem{preconj}{{\bf Conjecture}}

\newtheorem{preques}{{\bf Question}}

\newtheorem{prealphthm}{{\bf Theorem}}

\newtheorem{prepro}{{\bf Proposition}}

\newtheorem{preprb}{{\bf Problem}}

\newtheorem{prerem}{{\bf Remark}}

\def\conct[#1,#2]{\mbox {${#1} \leftrightarrow {#2}$}}
\def\dconct[#1,#2]{\mbox {${#1} \rightarrow {#2}$}}
\def\deg[#1,#2]{\mbox {$d_{_{#1}}(#2)$}}
\def\mindeg[#1]{\mbox {$\delta_{_{#1}}$}}
\def\maxdeg[#1]{\mbox {$\Delta_{_{#1}}$}}
\def\outdeg[#1,#2]{\mbox {$d_{_{#1}}^{^+}(#2)$}}
\def\minoutdeg[#1]{\mbox {$\delta_{_{#1}}^{^+}$}}
\def\maxoutdeg[#1]{\mbox {$\Delta_{_{#1}}^{^+}$}}
\def\indeg[#1,#2]{\mbox {$d_{_{#1}}^{^-}(#2)$}}
\def\minindeg[#1]{\mbox {$\delta_{_{#1}}^{^-}$}}
\def\maxindeg[#1]{\mbox {$\Delta_{_{#1}}^{^-}$}}

\def\dre[#1,#2,#3]{\mbox {${\cal E}_{_{#3}}(#1,#2)$}}
\def\pdre[#1,#2,#3]{\mbox {${\cal P}_{_{#3}}(#1,#2)$}}
\def\var[#1,#2]{\mbox {${\rm Var}_{_{#1}}(#2)$}}
\def\ls[#1]{\mbox {$\xi^{^{#1}}$}}
\def\hom[#1,#2]{\mbox {${\rm Hom}({#1},{#2})$}}
\def\onvhom[#1,#2]{\mbox {${\rm Hom^{v}}(#1,#2)$}}
\def\onehom[#1,#2]{\mbox {${\rm Hom^{e}}(#1,#2)$}}
\def\core[#1]{\mbox {$#1^{^{\bullet}}$}}
\def\cay[#1,#2]{\mbox {${\rm Cay}({#1},{#2})$}}
\def\Cay{{\rm Cay}}
\def\cays[#1,#2]{\mbox {${\rm Cay_{s}}({#1},{#2})$}}
\def\dirc[#1]{\mbox {$\stackrel{\rightarrow}{C}_{_{#1}}$}}
\def\cycl[#1]{\mbox {${\bf Z}_{_{#1}}$}}

\def\sdg[#1]{\mbox {$\stackrel{\leftrightarrow}{#1}$}}
\def\Ga{\Gamma}
\def\2sc{{2{\rm SCay}}}
\def\Aut{{\rm Aut}}
\def\Sym{{\rm Sym}}
\def\id{{\sf{id}}}
\def\vp{\varphi}
\begin{document}
\begin{center}
{\Large \bf Two-sided Cayley graphs}\\
\vspace*{0.5cm}
{\bf Moharram N. Iradmusa$^a$, Cheryl E. Praeger$^b$}\\
\vspace*{0.2cm}
{\it {\small $^{a}$Department of Mathematical Sciences, Shahid Beheshti University, Tehran, Iran\\
iradmusa@ipm.ir\\
$^{b}$School of Mathematics and Statistics, The University of Western Australia,\\
35 Stirling Highway, Crawley, WA 6009, Australia\\
cheryl.praeger@uwa.edu.au\\
Also affiliated with King Abdulaziz University, Jeddah, Saudi Arabia}}\\
%\textbf{\hrulefill}\\
\end{center}
\begin{abstract}
\noindent 
We introduce a family of graphs that generalises the class 
of Cayley graphs. For non-empty subsets $L,R$ of a group $G$, 
the two-sided Cayley graph $\2sc(G;L,R)$ is the directed graph 
with vertex set $G$ and an arc from $x$ to $y$ if and only if $y=\ell^{-1}xr$ for 
some $\ell\in L$ and $r\in R$. Thus, in common with Cayley graphs, 
two-sided Cayley graphs may be useful to model networks as 
the same routing and communication scheme can be implemented at each vertex.
We determine when two-sided Cayley graphs are simple undirected graphs, and
give sufficient conditions for them to be connected, vertex-transitive, 
or Cayley graphs. Several open problems are posed. 
Many examples are given, including one on 12 vertices with connected 
components of sizes 4 and 8.
\\
\begin{itemize}
\item[]{{\footnotesize {\bf Key words:}\ Cayley graph, vertex-transitive graph, Group.}}
\item[]{ {\footnotesize {\bf Subject classification: 05C25} .}}
\end{itemize}
\end{abstract}
%%%%%%%%%%%%%%%%%%%%%%%%%%%%%%%%%%%%%%%%%%%%%%%%%%%%%%%%%%%%%%%%%%%5
\section{Introduction}
%\cite{She89a}

In this paper we study a family of graphs that generalises the class 
of Cayley graphs. We call the new graphs two-sided Cayley graphs and 
explore some of their properties. For a subset $S$ of a group $G$ such 
that the identity $e\notin S$ and $S=S^{-1}$ (where $S^{-1}=\{s^{-1}|
s\in S\}$), the \emph{Cayley graph} $\Ga=\Cay(G,S)$ is the graph with 
vertex set $G$ such that a vertex pair $(x, y)$ is an arc if and only 
if $xy^{-1}\in S$. The conditions $e\not\in S$ and $S=S^{-1}$ ensure 
that $\Ga$ may be viewed as a simple undirected graph, and that the 
group $G$ acts by right multiplication as a vertex-transitive group 
of automorphisms \cite{bei04a}. Cayley graphs were introduced by Arthur 
Cayley in 1878 to explain the concept of an abstract group given by a set 
of generators \cite{cay78a}. There are many meaningful applications of 
Cayley graphs in molecular biology, computer science and coding theory
\cite{baf96a, hah97a, kec95a,lak93a}. Because of their symmetry properties, 
Cayley graphs are used as models for many interconnection networks. Moreover 
various generalisations of Cayley graphs have been introduced in the literature 
as prototypes of transitive graphs with some degree of success: for instance, 
quasi-Cayley graphs by Gauyacq \cite{gau97a},  various kinds of groupoid 
graphs by Mwambene  \cite{mwa01a, mwa09a}, group action graphs by Annexstein 
et al \cite{ann90a}, and  general semigroup graphs by Kelarev and the 
second author \cite{kel03a}. 

One of the advantages of using Cayley graphs as models of networks is 
that their vertex-transitivity makes it possible to implement the same 
routing and communication scheme at each node of the network. At each 
vertex $x$ of $\Cay(G,S)$ the edges from $x$ go to the vertices $sx$ 
for $s\in S$. Thus we may view $S$  as a `connection subset' $\hat S:=
\{\lambda_{e,s}\mid s\in S\}$ of permutations of $G$ (where $\lambda_{e,s}$ 
is the permutation $g\mapsto gs$) that determines adjacency locally at each vertex.   

\subsection*{Two-sided Cayley graphs: basic properties}

A two-sided Cayley graph also possesses this local connection property, and is more general in the sense that the connection subset of permutations may act \emph{on both sides}, that is to say, we consider permutations of the form $\lambda_{\ell,r}: g\mapsto \ell^{-1}gr$ for certain $\ell,r\in G$.

\begin{definition}\label{def:2sc}
{\rm
For non-empty subsets $L,R$ of a group $G$, we define the \emph{two-sided 
Cayley graph} $\2sc(G;L,R)$ as the (directed) graph with vertex
set $G$ with an arc from $x$ to $y$ if and only if $y=\ell^{-1}xr$ for 
some $\ell\in L$ and $r\in R$. The \emph{connection set} of $\2sc(G;L,R)$ 
is defined as the set $\hat{S}(L,R)=\{\lambda_{\ell,r}: (\ell,r)\in L\times R\}$.
}
\end{definition}

We note that the edges of $\2sc(G;L,R)$ from a vertex $x$ go to the vertices 
$(x)\lambda$, for $\lambda\in\hat S(L,R)$. If the adjacency relation is symmetric, 
in the sense that $(x,y)$ is an arc if and only if $(y,x)$ is an arc, then 
$\2sc(G;LR)$ will be regarded as an undirected graph. Conditions for this to 
occur and also for the graph to contain no loops or multiple edges, that is to say, conditions 
for  $\2sc(G;LR)$ to be a simple undirected graph, can be given in terms of 
properties of  $L$ and $R$ as follows.

\begin{definition}\label{def:2scprop}
{\rm
A pair $(L,R)$ of subsets of a group $G$ with identity element $e$  
has the \emph{$2$S-Cayley property} if both $L, R$ are non-empty, 
and the following conditions all hold.
\begin{itemize}
\item[(1)] $L^{-1}gR=LgR^{-1}$ for each $g\in G$;
\item[(2)] $L^g\cap R=\varnothing$ for each $g\in G$;
\item[(3)] $(L\,L^{-1})^g\cap (R\,R^{-1})=\{e\}$ for each $g\in G$.
\end{itemize} 
}
\end{definition} 

\begin{theorem}\label{thm:2scsimple}
Let $L, R$ be nonempty subsets of a group $G$. Then $\2sc(G;L,R)$ is a simple undirected 
graph if and only if $(L,R)$ has the $2$S-Cayley 
property. If this is the case then $\2sc(G;L,R)$ has valency $|L|.|R|$.
\end{theorem}

We prove Theorem~\ref{thm:2scsimple} in Section~\ref{sec:props} by 
studying separately the significance of each of the conditions in 
Definition~\ref{def:2scprop} (Lemma~\ref{lem:2scsimple}). 
If both of the sets $L,R$ are inverse-closed, that is to say, if $L=L^{-1}$ and $R=R^{-1}$,
then the conditions of Definition~\ref{def:2scprop} are somewhat simpler. However 
it is possible for a pair of subsets to have the $2$S-Cayley property but not both
be inverse-closed, as demonstrated by Example~\ref{ex0}. 

% \begin{problem}\label{prob2}
% Decide whether or not Definition~\ref{def:2scprop}~(1) 
% is equivalent to the subsets $L,R$ being inverse-closed.
% \end{problem}

We also show in Lemma~\ref{lem:2scsimple} that, for $(L,R)$ with the $2$S-Cayley property, 
each element $\lambda_{\ell,r}$ of $\hat S(L,R)$ is a 
\emph{derangement} of $G$, that is to say, $\lambda_{\ell,r}$
has no fixed points in $G$. Proposition~\ref{prop:derangement} provides
a set of conditions on the connection set $\hat S(L,R)$ which 
hold for $(L,R)$ with the $2$S-Cayley property.

\begin{remark}\label{rem:2sc}{\rm
(a) Each Cayley graph is a two-sided Cayley graph, namely if 
$e\not\in S$ and $S=S^{-1}$, then $\Cay(G,S)=\2sc(G;S,\{e\})$, 
and it is easy to check that $(S,\{e\})$ has the 2S-Cayley 
property. There are sometimes other possibilities for the 
subsets $L, R$ giving the same Cayley graph. We describe 
such a graph in Example~\ref{ex1}.

\medskip\noindent
(b)	As mentioned above each Cayley graph is vertex transitive. 
This is not the case for two-sided Cayley graphs, indicating in 
particular that these graphs form a strictly larger family of graphs: 
we present in Example~\ref{ex2} a two-sided Cayley graph with two 
non-isomorphic connected components. 

\medskip\noindent
(c) A Cayley graph $\Cay(G,S)$ is connected if and only if $S$ 
generates $G$. Theorem~\ref{thm:conn} gives an analogue of this 
connectivity criterion for two-sided Cayley graphs in the case where $L, R$
are inverse-closed subsets.
}
\end{remark}

By a \emph{word} in a subset $L$ of a group $G$, we mean a string 
$w=\ell_1\ell_2\dots\ell_k$ with each $\ell_i\in L$; the integer 
$k$ is called the \emph{length} of $w$, denoted $|w|$, and we often 
identify $w$ with its \emph{evaluation} in $G$ (the element of $G$ 
obtained by multiplying together the $\ell_i$ in the given order). 

\begin{theorem}\label{thm:conn}
Let  $L, R$ be inverse-closed subsets of a group $G$ with the $2$S-Cayley property, 
and let $\Gamma=\2sc(G;L,R)$. Then $\Gamma$ is connected if and only if
\begin{itemize}
\item[$(*)$] $G=\langle L\rangle\,\langle R\rangle$, and there 
exist words $w$ in $L$ and $w'$ in $R$, with lengths of opposite 
parity, such that $ww'=e$ in $G$. 
\end{itemize}
Further, if  $G=\langle L\rangle\,\langle R\rangle$, but condition 
$(*)$ does not hold, then $\Gamma$ is disconnected with exactly 
two connected components. 
\end{theorem}

\begin{remark}\label{rem:conn}
{\rm
$(a)$\ The exceptional case of Theorem~\ref{thm:conn} does arise, see the 
family of graphs in Example~\ref{ex3}. Moreover, the proof of 
Theorem~\ref{thm:conn} in Section~\ref{sec:props} shows that,  
if  $G=\langle L\rangle\,
\langle R\rangle$ but condition $(*)$ does not hold then, for each $g\in G$, 
and each expression $g=ww'$, where $w$ and $w'$ are words in $L$ and 
$R$ respectively, the parity of the sum $|w|+|w'|$ is independent of the
words $w, w'$, and depends only on $g$. Let $\delta(g)\in\{0,1\}$ where 
$\delta(g)\equiv |w|+|w'|\pmod{2}$. We show that the connected components  
of $\Gamma$ are the sets $\mathcal{C}_\delta=\{g\mid g\in G,\ \delta(g)=\delta\}$, for $\delta\in\{0,1\}$.
In particular $e\in\mathcal{C}_0$ and $L\cup R\subseteq\mathcal{C}_1$. 

\medskip\noindent$(b)$\ It is not clear how to modify the proof to cover the 
cases where not both of $L, R$ are inverse-closed. 
}
\end{remark}

\begin{problem}\label{probconn}{
 Find necessary and sufficient conditions for  $\2sc(G;L,R)$
to be connected, where $L,R$ have the $2S$-Cayley property but are not both inverse-closed.
}
\end{problem}

We now mention several generic isomorphisms between two-sided Cayley graphs for a group $G$,
that are induced from automorphisms of $G$. This result is proved in Section~\ref{sec:iso}.

\begin{theorem}\label{thm:iso}
Let  $L, R$ be non-empty subsets of a group $G$, let  $x,y\in G$, and 
$\sigma\in \Aut(G)$.
\begin{enumerate}
\item[(a)] If at least one of the following pairs has the  $2$S-Cayley property,
then they all do:
$(L,R),\quad (R,L),\quad (L^{\sigma},R^{\sigma}),\quad (L^{x},R^{y})$.
\item[(b)] Let $\Gamma=\2sc(G;L,R)$ and suppose that $(L,R)$ 
has the $2$S-Cayley property. Then
\[
\2sc(G;L,R)\cong \2sc(G;R,L)\cong \2sc(G;L^{\sigma},R^{\sigma})\cong \2sc(G;L^x,R^y).
\]
\end{enumerate}
\end{theorem}

The smallest vertex transitive graph that is not a Cayley graph is 
the Petersen graph $P$ on $10$ vertices, and we wondered whether $P$ could be a 
two-sided Cayley graph for some group of order $10$. However, 
the following result about connected two-sided Cayley graphs of prime valency 
shows that $P$ cannot be a two-sided Cayley graph, since $P$ has valency 3 
and is not a Cayley graph, and the only groups of order 10 are abelian or dihedral.

\begin{theorem}\label{cor:primval}
Let  $L, R$ be subsets of a group $G$ with the $2$S-Cayley property, 
and suppose that $\Gamma=\2sc(G;L,R)$ is connected and regular of 
prime valency $p$.  
Then either $\Ga$ is a Cayley graph, or $p$ is odd and the 
following all hold, up to interchanging $L$ and $R$.
\begin{enumerate}
\item[(a)] $R=\{r\}$, where $r$ lies in a $G$-conjugacy 
class $\mathcal{C}$ of non-central elements; 
\item[(b)] $G=\langle L\rangle$,  $|L|=p$, and
$L\cap\mathcal{C}=\varnothing$. 
%and $L$ contains an element $\ell$ such that $\ell^2=e$.
\end{enumerate}
Moreover if $G$ is abelian or if $G$ is dihedral of twice-odd
order, then $\Ga$ is a Cayley graph.
\end{theorem}

\begin{problem}\label{q1}
Find other natural classes of two-sided Cayley graphs that 
can be guaranteed to be Cayley graphs.
\end{problem}

\subsection*{Symmetry of two-sided Cayley graphs}

Example~\ref{ex2} demonstrates that not all two-sided 
Cayley graphs are Cayley graphs. We next give sufficient 
conditions on $L, R$ for $\2sc(G;L,R)$ to be a Cayley graph for $G$. 

It has been known since at least 1958 that a simple undirected graph is a 
Cayley graph if and only if its automorphism group has a subgroup 
acting regularly on the vertices (see Sabidussi \cite[Lemma 4]{MR0097068}). 
A permutation group is \emph{regular} if it is transitive, and only the 
identity fixes a point. For any group $G$, the 
\emph{right regular representation} and \emph{left regular representation}
give two regular subgroups of $\Sym(G)$, each isomorphic to $G$, namely
$G_R=\{\lambda_{e,g}|g\in G\}$ and $G_L=\{\lambda_{g,e}|g\in G\}$.

\begin{theorem}\label{twosidedisCayley}
Let  $L, R$ be subsets of a group $G$ with the $2$S-Cayley property, 
and let $\Gamma=\2sc(G;L,R)$.   
\begin{enumerate}
\item[(a)]  $G_R\leq\Aut(\Ga)$ if and only if $L^{-1}gR=L^{-1}Rg$ for each $g\in G$; here $\Ga=\Cay(G,L^{-1}R)$.
\item[(b)]  $G_L\leq\Aut(\Ga)$ if and only if $L^{-1}gR=gL^{-1}R$ for each $g\in G$; here $\Ga\cong\Cay(G,R^{-1}L)$.
\end{enumerate}
\end{theorem}

\begin{remark}\label{rem:isCayley}{\rm 
(a) Theorem~\ref{twosidedisCayley} is proved in 
Section~\ref{sec:cay}, and we derive from it, in Proposition~\ref{twosidedisCayley1},
two sufficient conditions for $\2sc(G;L,R)$ to be a Cayley graph 
for $G$ in terms of certain factorisations of the group $G$.

\medskip\noindent (b) The conditions given in Theorem~\ref{twosidedisCayley}
are far from necessary. We give in Example~\ref{ex4} a family of examples of 
two-sided Cayey graphs, all of which are Cayley graphs, where none of the
conditions of Theorem~\ref{twosidedisCayley} or of
Proposition~\ref{twosidedisCayley1} hold. 

\medskip\noindent (c) A more general study of the permutations 
$\lambda_{x,y}$ which act as automorphisms of a two-sided Cayley graph
leads to a sufficient condition for such a graph to be 
vertex-transitive, namely that $G$ factorises as $G=N_G(L)N_G(R)$ (Proposition~\ref{vt}). 
Studying the subclass of two-sided Cayley graphs $\Gamma=\2sc(G;L,R)$ 
which are guaranteed to be vertex-transitive by the condition
 $G=N_G(L)N_G(R)$ may be profitable. In particular the $2S$-Cayley 
property becomes much simpler (Proposition~\ref{simple}).  
However we do not know if all such graphs are Cayley graphs,
or if there are vertex-transitive two-sided Cayley graphs 
for which $G\ne N_G(L)N_G(R)$.
}
\end{remark}

\begin{problem}\label{prob:simple}
Determine whether or not all two-sided Cayley graphs
$\Gamma=\2sc(G;L,R)$ for which  $G=N_G(L)N_G(R)$
are Cayley graphs.
\end{problem}

\begin{problem}\label{prob:nas}
Find necessary and sufficient conditions on a two-sided Cayley graph 
to be vertex-transitive. 
\end{problem}

%%%%%%%%%%%%%%%%%%%%%%%%%%%555555555555555555555555555

\section{Examples of two-sided Cayley graphs}\label{sec:ex}

In this section we give several examples of two-sided Cayley graphs 
which exhibit various interesting properties. 
First we give an example of a two-sided Cayley graph which is a 
simple graph but for which the defining subsets $L, R$ are not inverse-closed.

\begin{example}\label{ex0}{\rm
Let $G$ be the dihedral group $\langle a,b\ |\ a^6=b^2=e,bab=a^{-1}\rangle$ of
order $12$, and let $L=\{a,a^2\}$ and $R=\{b, a^3b\}$. We note that $L$ is 
not inverse-closed, but $R=R^{-1}$. We check below that 
$(L,R)$ has the 2S-Cayley property, and so by Theorem~\ref{thm:2scsimple}, 
$\Ga=\2sc(G;L,R)$ is a simple undirected two-sided Cayley graph of valency $4$. 
Also the condition of Theorem~\ref{twosidedisCayley}~(b) holds so the 
left multiplication action of $G_L$ is a subgroup of automorphisms, proving
that  $\Ga$ is a Cayley graph, in fact $\Ga=\Cay(G,LR)$.
Also $\Ga$ is the lexicographic product $C_6[2.K_1]$ of a cycle 
of length $6$ and two isolated vertices $2.K_1$. 
}
\end{example}

\begin{proof}{
We verify the condition of Definition~\ref{def:2scprop}
for $\Ga=\2sc(G;L,R)$. Note that $R=R^{-1}$. For $g=a^ib$ we have $gR=\{a^i,a^{i+3}\}$
and so $L^{-1}gR=LgR^{-1}=\{a^j | j=1,2,4,5\}$. Similarly for $g=a^i$, 
$L^{-1}gR=LgR^{-1}=\{a^jb \,|\, j=1,2,4,5\}$. This proves condition (1). 
Next we note that, for each $g\in G$, $L^g$ equals $L$ or $L^{-1}$ and each 
is disjoint from $R$, so (2) holds. Finally $L L^{-1}=\{e,a,a^5\}$ is $G$-invariant
and meets $R R^{-1}=\{e, a^3\}$ in $\{e\}$, so (3) holds. 
}
\end{proof}

Next we give the Cayley graph 
example referred to in Remark~\ref{rem:2sc}(a).

\begin{example}\label{ex1}{\rm
Let $G=\Sym(3)$, $L=\{e,(12)\}$ and $R=\{(123),(132)\}$. It is easy 
to check that $(L,R)$ satisfies conditions $(1)-(3)$ of Definition~\ref{def:2scprop}, 
so $(L,R)$ has the 2S-Cayley property, and by Theorem~\ref{thm:2scsimple}, 
$\Ga:=\2sc(G;L,R)$ is a simple undirected two-sided Cayley graph of valency 
$4$. A straightforward computation verifies that $\Ga$  is isomorphic to the 
complete graph $K_6$ with the edges of the following matching removed: 
$\{e,(12)\}, \{(13),(123)\}, \{(23),(132)\}$. Also $\Ga$ is the Cayley graph 
$\Cay(G,LR)$ (see Theorem~\ref{twosidedisCayley} and Remark~\ref{rem:isCayley}(a)). 
This representation of $\Cay(G,LR)$ as a two-sided Cayley graph is in addition to 
the one in Remark~\ref{rem:2sc}, namely $\2sc(G;LR,\{e\})$. Note also that $\Ga
\cong\2sc(G;R,L)=\Cay(G,RL)$, see  Theorem~\ref{twosidedisCayley}.
}
\end{example}

Now we present a two-sided Cayley graph with non-isomorphic connected components, 
promised in Remark~\ref{rem:2sc}(b). 

\begin{example}\label{ex2}{\rm
Let $G$ be the dihedral group $\langle a,b\ |\ a^6=b^2=e,bab=a^{-1}\rangle$ of
order $12$, and let $L=\{ab,a^3,e\}$ and $R=\{b\}$. It is easy to check that 
$(L,R)$ has the 2S-Cayley property, and so by Theorem~\ref{thm:2scsimple}, 
$\2sc(G;L,R)$ is a simple undirected two-sided Cayley graph of valency $3$. 
In addition, $G\ne \langle  L\rangle\,\langle R\rangle$, so  $\2sc(G;L,R)$ 
is disconnected,  by Theorem~\ref{thm:conn}. In fact  $\2sc(G;L,R)$ has two 
non-isomorphic components of orders 4 and 8, as shown in Figure~\ref{pic1}, 
and in particular this graph is not vertex-transitive.
}
\end{example}

\begin{figure}
\begin{center}
\includegraphics[height=3cm]{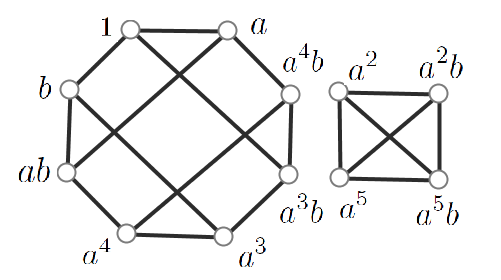}
\caption{Disconnected two-sided Cayley graph with non-isomorphic components}
\label{pic1}
\end{center}
\end{figure}

The next family of examples demonstrates that disconnected two-sided Cayley graphs
$\2sc(G;L,R)$ exist with  $G= \langle  L\rangle\,\langle R\rangle$ 
(see Remark~\ref{rem:conn}).

\begin{example}\label{ex3}{\rm
Let $G$ be the dihedral group $\langle a,b\ |\ a^n=b^2=e,bab=a^{-1}\rangle$ of
order $2n$, and let $L=\{b\}$ and $R=\{a,a^{-1}\}$. It is easy to check that 
$(L,R)$ has the 2S-Cayley property, and so by Theorem~\ref{thm:2scsimple}, 
$\Ga:=\2sc(G;L,R)$ is a simple undirected two-sided Cayley graph of valency $2$,
so each connected component of $\Ga$ is a cycle. 
An easy computation confirms that if $n$ is odd then $\Gamma = C_{2n}$ is connected, 
while if $n$ is even then $\Gamma=2.C_n$ has two connected components.

We now see how this confirms the conclusions of Theorem~\ref{thm:conn}.
For all $n$, $G=\langle L\rangle \,\langle R\rangle$.  Let $w,w'$ be words 
in $L, R$ respectively such that $ww'=e$. In particular $ww'\in\langle a\rangle$,
and hence $w$ evaluates to $e$ in $G$, so $|w|$ is even. Thus also $w'=e$ in $G$, and there exists
such a word $w'$ of odd length if and only if $n$ is odd, so by 
Theorem~\ref{thm:conn}, $\Ga$ is connected if and only if $n$ is odd.
Suppose that $n$ is even, and note that the identity $e$ is joined to $ba$ and 
$ba^{-1}$. In this case the two components of $\Ga$ are as described in 
Remark~\ref{rem:conn}, for example, the component $\mathcal{C}_0$ containing $e$
consists of the elements $a^{2i}, ba^{2i-1}$ for all $i$.   
}
\end{example}

Finally we give the examples referred to in Remark~\ref{rem:isCayley}~(b).
To construct these graphs we use groups with at least one non-normal 
subgroup. A \emph{Hamiltonian group} is a non-abelian group in which all subgroups are 
normal. Richard Dedekind investigated finite Hamiltonian groups in 1895, and came
close to a classification of them  \cite{MR1510943}.
The classification was completed by Baer \cite{MR1556916}. They are
precisely groups of the form $Q_8\times P$, where $Q_8$ is the quaternion group of 
order $8$ and $P$ is any abelian group containing no element of order $4$.
Dedekind named these groups in honour of William Hamilton, the discoverer 
of quaternions, and groups in which all subgroups are normal (that is,  
abelian or Hamiltonian groups) are sometimes called \emph{Dedekind groups}, or \emph{quasi-Hamiltonian 
groups}. These are the only groups we need to avoid in Example~\ref{ex4}.

\begin{example}\label{ex4}{\rm
Let $G_1$ and $G_2$ be non-abelian, non-Hamiltonian groups 
with identity elements $e_1$ and $e_2$, and set $G=G_1\times G_2$. For $i\in \{1,2\}$, let $H_i$ be a 
non-normal subgroup of $G_i$, and set $H_i^\#:= H_i\setminus\{e_i\}$. Let
$\Gamma=\2sc(G;L,R)$, where $L=H_1^\#\times \{e_2\}$ and $R=\{e_1\}\times H_2^\#$. 
}
\end{example}

\begin{lemma}\label{lem:ex4}
Let $G, L, R, \Ga$ be as in Example~$\ref{ex4}$. Then 
\begin{enumerate}
 \item[(a)] $(L,R)$ has the $2S$-Cayley property,
but does not satisfy any of the properties on $L, R$ in Theorem~$\ref{twosidedisCayley}$
or Proposition~$\ref{twosidedisCayley1}$.
\item[(b)]  The map $\vp:(g_1,g_2)\rightarrow (g_1,g_2^{-1})$ induces an isomorphism from $\Ga$ to 
$\Cay(G, LR)$.   
\end{enumerate}
\end{lemma}

\begin{proof}{
The sets $L, R$ are inverse-closed, so Definition~\ref{def:2scprop}(1) holds. The other
two conditions Definition~\ref{def:2scprop}(2) and~(3) hold because $L\subset G_1
\times\{e_2\}$ while $R\subset\{e_1\}\times G_2$. Thus $(L,R)$ has the $2S$-Cayley property.

If $g=(e_1,g_2)$ with $g_2\in G_2\setminus N_{G_2}(H_2)$, then $L^{-1}gR\ne L^{-1}Rg$ (such an
element $g_2$ exists since $H_1$ is not normal in $G_1$).
Similarly if $g=(g_1,e_2)$ with $g_1\in G_1\setminus N_{G_1}(H_1)$, then $L^{-1}gR\ne gL^{-1}R$,
so neither condition of Theorem~$\ref{twosidedisCayley}$ holds. 
The factors on the left hand side of condition $(\delta)$ in Proposition~\ref{twosidedisCayley1}
are $N_{G_1}(H_1)\times N_{G_2}(H_2)$ and $G_1\times N_{G_2}(H_2)$, and their product is not equal to $G$.
Similarly condition $(\delta')$ does not hold. This proves part (a).

Each arc of $\Ga$ is of the form $((x,y), (\ell^{-1}x,yr))$,
for some $(\ell,e_2)\in L, (e_1,r)\in R$. This is mapped by $\vp$ to the pair $((x,y^{-1}),
(\ell^{-1}x,r^{-1}y^{-1}))$, which is an arc of $\Cay(G, LR)$, and conversely each arc of $\Cay(G, LR)$
is the image under $\vp$ of an arc of $\Ga$. 
}
\end{proof}

%%%%%%%%%%%%%%%%%%%%%%%%%%%%%%%%%%%%%%%%%%%%%%%%%%%%%%%
\section{Basic properties of two-sided Cayley graphs}\label{sec:props}

In this section we prove Theorems~\ref{thm:2scsimple}, and~\ref{thm:conn}. 
We begin by showing the significance of each of
the conditions in Definition~\ref{def:2scprop}. Recall the definition of $\lambda_{\ell,r}$ 
given before Definition~\ref{def:2sc}. We will say that $\Ga=\2sc(G;L,R)$ contains
\emph{multiple arcs} if, for some $x,y\in G$, there are at least two arcs of $\Ga$ from $x$ to $y$. 

\begin{lemma}\label{lem:2scsimple}
Let $L, R$ be nonempty subsets of a group $G$, and let $\Ga=\2sc(G;L,R)$.
\begin{itemize}
\item[(a)] Definition~$\ref{def:2scprop}(1)$ holds if and only if the adjacency 
relation of $\Ga$ is symmetric. 
\item[(b)] Definition~$\ref{def:2scprop}(2)$ holds if and only if 
each $\lambda_{\ell,r}$ (for $\ell\in L, r\in R$) is a derangement, and this 
holds if and only if $\Ga$ has no loops.
\item[(c)] Definition~$\ref{def:2scprop}(3)$ holds if and only if $\Ga$ has no multiple arcs.
\end{itemize} 
In particular the conclusion of Theorem~$\ref{thm:2scsimple}$ is valid. 
\end{lemma}

\begin{proof}{ (a) 
Suppose first that  the adjacency relation of $\Ga$ is symmetric. 
Then we have the following equivalent conditions on elements $x,y\in G$.

\begin{center}
\begin{tabular}{ll} 
 & $y=\ell_1^{-1} x r_1$ for some $\ell_1\in L, r_1\in R$\\
iff & $(x,y)$ is an arc of $\2sc(G;L,R)$\\
iff & $(y,x)$ is an arc of $\2sc(G;L,R)$\\
iff & $x=\ell_2^{-1} y r_2$, or equivalently $y=\ell_2 x r_2^{-1}$, for some $\ell_2\in L, r_2\in R$\\
\end{tabular}
\end{center}

\noindent
It follows that $L^{-1}xR=LxR^{-1}$ for all $x\in G$, that is,
Definition~$\ref{def:2scprop}(1)$ holds. Conversely, 
suppose that this condition holds, and that $(x,y)$ is an arc of $\2sc(G;L,R)$.
Then $y=\ell^{-1}xr$ for some  $\ell\in L, r\in R$, and since   $L^{-1}xR=LxR^{-1}$
we also have  $y=\ell_2 x r_2^{-1}$, for some $\ell_2\in L, r_2\in R$, so that $(y,x)$ is also an arc. 

(b) If $L^g\cap R\neq\varnothing$ for some $g\in G$, then $g^{-1}\ell g=r$, 
or equivalently, $(g)\lambda_{\ell,r}=\ell^{-1}gr=g$, for some  $\ell\in L, r\in R$, 
which shows that $\lambda_{\ell,r}$ fixes the vertex $g$ and
that there is a loop on $g$ in $\Ga$. Conversely existence of a loop on $g$
implies that $\ell^{-1}gr=g$, for some  $\ell\in L, r\in R$, and hence that
$g^{-1}\ell g=r\in L^g\cap R$. Similarly, if some $\lambda_{\ell,r}$ fixes a vertex 
$g$, then we have $\ell^{-1}gr=g$ and $g^{-1}\ell g=r\in L^g\cap R$.

(c)  Suppose that Definition~$\ref{def:2scprop}(1)$ holds. Assume first that
 $(L\,L^{-1})^g\cap (R\,R^{-1})\ne\{e\}$ for some $g\in G$. Then $x:=g^{-1}\ell_1\ell_2^{-1}g=r_1r_2^{-1}
\ne e$ for some $\ell_1,\ell_2\in L$ and $r_1,r_2\in R$, and hence $y:=\ell_2^{-1}g r_2
=\ell_1^{-1}g r_1$. Since $r_1r_2\ne e$, the pairs $(\ell_1,r_1), 
(\ell_2, r_2)\in L\times R$ are distinct and so there are two arcs in $\Ga$ 
from  $g$ to $y$. Conversely suppose that $\Ga$ contains multiple arcs, say two distinct
arcs from vertex $x$ to $y$. Then we have $y=\ell_i^{-1}xr_i$
for distinct pairs $(\ell_i,r_i)\in L\times R$ with $i=1,2$. Now $r_1\ne r_2$ since 
otherwise we would also have $\ell_1 = xr_1y^{-1} = xr_2y^{-1}= \ell_2$.
Hence 
\[
e\ne r_2r_1^{-1}=(x^{-1}\ell_2 y)(y^{-1}\ell_1^{-1}x)=x^{-1}(\ell_2\ell_1^{-1})x\in (L\,L^{-1})^x\cap (R\,R^{-1})
\]
so condition Definition~$\ref{def:2scprop}(3)$ fails.

Finally we prove Theorem~\ref{thm:2scsimple}. If $\Ga$ is simple and undirected, 
then Definition~$\ref{def:2scprop}$ parts (1), (2) and (3) follow from parts (a), (b) 
and (c) above, respectively, so $(L,R)$ has the $2$S-Cayley property. Conversely if
$(L,R)$ has the $2$S-Cayley property, then parts (a)--(c) imply that
$\Ga$ is simple and undirected. In this case edges from a vertex
$x$ go to the vertices $\ell^{-1}xr$ for $(\ell,r)\in (L,R)$, and if 
$\ell_1^{-1}xr_1=\ell_2^{-1}xr_2$, then $x^{-1}(\ell_2\ell_1^{-1})x
=r_2r_1^{-1}\in (L\,L^{-1})^x\cap (R\,R^{-1})$, so by condition (3),
this element is the identity, and so $r_2=r_1$ and $\ell_2=\ell_1$.
Thus each vertex of $\Ga$ is joined to exactly $|L|.|R|$ other vertices.
}
\end{proof}

Next we consider various conditions on the connection set $\hat S(L,R)$
of $\2sc(G:L,R)$ defined in Definition~\ref{def:2sc}. 

\begin{lemma}\label{lem:lambda1}
Let  $L, R$ be nonempty subsets of a group $G$. Then
\begin{enumerate}
 \item[(a)] $\lambda_{\ell,r}^{-1}=\lambda_{\ell^{-1},r^{-1}}$, for all $\ell, r\in G$, and in particular $\hat S(L,R)^{-1}=\hat S(L^{-1},R^{-1})$.
 \item[(b)] $\lambda_{\ell,r}=\lambda_{u,v}$ if and only if there exists $x\in Z(G)$ such that
$u=x\ell$ and $v=xr$.
\end{enumerate}
\end{lemma}
  
\begin{proof}{
(a) By definition, the inverse of $\lambda_{\ell,r}$ maps an element 
$x$ to $\ell xr^{-1}$, and hence is equal to $\lambda_{\ell^{-1},r^{-1}}$.

(b) Suppose that $u=x\ell$ and $v=xr$ where $x\in Z(G)$. Then 
$\lambda_{u,v}$ maps $g$ to $u^{-1}gv = \ell^{-1}x^{-1}gxr=
\ell^{-1}gr$ so $\lambda_{u,v}=\lambda_{\ell,r}$.
Suppose conversely that $\lambda_{u,v}=\lambda_{\ell,r}$.
Considering the action on the identity $e$, we have $\ell^{-1}r=u^{-1}v$, so
$x:=u\ell^{-1}=vr^{-1}$. Then, considering the action on 
an arbitrary $g\in G$ we have $\ell^{-1}gr=u^{-1}gv=\ell^{-1}x^{-1}gxr$,
whence $g=x^{-1}gx$ for all $g\in G$ and so $x\in Z(G)$. 
}
\end{proof}

\begin{proposition}\label{prop:derangement}
Let $L, R$ be subsets of a group $G$ such that $(L,R)$ has 
the $2$S-Cayley property, and let $D(G)$ denote the set of 
all derangements of $G$. Then the following hold.
\begin{enumerate}
 \item[$(\hat 1)$] Given $\ell\in L, r\in R$, for each $x\in G$ there exist
unique $\ell'\in L, r'\in R$ such that $\varphi:=\lambda_{\ell,r}\circ\lambda_{\ell',r'}$ fixes $x$. 
Moreover $\varphi={\rm id}$ if and only if $L\ell\cap Rr\cap Z(G)\ne\varnothing$. %  $\hat{S}(L,R)=\hat{S}(L,R)^{-1}$;
\item[$(\hat 2)$] $\hat{S}(L,R)\subseteq D(G)$;
\item[$(\hat 3)$]  $(\hat{S}(L,R)\circ \hat{S}(L,R)^{-1})\setminus \{\id\}\subseteq D(G)$. 
\end{enumerate}
\end{proposition}

\begin{proof}
{Let  $\ell\in L, r\in R$, and $x\in G$. 
First we prove that $|L^{-1}xR|=|L|\,|R|$. 
Suppose that $\ell_1^{-1}xr_1=\ell_2^{-1}xr_2$,
where the $\ell_i\in L$ and the $r_i\in R$.
Then $x^{-1}(\ell_2\ell_1^{-1})x=r_2r_1^{-1}\in (L\,L^{-1})^x\cap(R\,R^{-1})$,
which is equal to $\{e\}$ by  Definition~$\ref{def:2scprop}(3)$. 
It follows that distinct pairs $(\ell,r)$ from $L\times R$ give distinct 
elements $\ell^{-1}xr$, so $|L^{-1}xR|=|L|\,|R|$.
Now, by Definition~$\ref{def:2scprop}(1)$, $L^{-1}xR
=LxR^{-1}$, and hence there exist $\ell'\in L, r'\in R$ such
that $\ell^{-1}xr=\ell'x(r')^{-1}$. This implies that  $\varphi:=
\lambda_{\ell,r}\circ\lambda_{\ell',r'}$ fixes $x$. The uniqueness of 
$\ell', r'$ follows from the fact that $|LxR^{-1}|=|L^{-1}xR|=|L|\,|R|$.

If $\varphi= {\rm id}$ then $\lambda_{\ell',r'}=\lambda_{\ell,r}^{-1}$,
which is $\lambda_{\ell^{-1},r^{-1}}$, by Lemma~\ref{lem:lambda1}(a). Then, by 
Lemma~\ref{lem:lambda1}(b), $\ell'=z\ell^{-1}$ and $r'=zr^{-1}$, 
for some $z\in Z(G)$.  Hence $z=\ell'\ell=r'r\in L\ell\cap Rr\cap Z(G)$. 
Conversely if $z\in L\ell\cap Rr\cap Z(G)$, then there exist $\ell'\in L, 
r'\in R$ such that $z=\ell'\ell=r'r$. By Lemma~\ref{lem:lambda1}, this implies that
$\lambda_{\ell',r'}=\lambda_{\ell^{-1},r^{-1}}=\lambda_{\ell,r}^{-1}$, and hence that the map 
$\varphi:= \lambda_{\ell,r}\circ\lambda_{\ell',r'}$ is the identity map.
This proves condition $(\hat{1})$.

Condition $(\hat 2)$ follows from the condition of
Definition~$\ref{def:2scprop}(2)$ and Lemma~\ref{lem:2scsimple}(b).
We now consider condition~$(\hat 3)$. 
Supppose that $\varphi:=\lambda_{\ell_1,r_1}\circ \lambda_{\ell_2^{-1},r_2^{-1}}\neq \id$,
where the $\ell_i\in L, r_i\in R$. Suppose for a contradiction that 
$\varphi$ fixes the element $g\in G$. Then
$g=(g)\varphi = \ell_2(\ell_1^{-1}gr_1)r_2^{-1}$, and hence
$r_1r_2^{-1}=g^{-1}(\ell_1\ell_2^{-1})g\in(L\,L^{-1})^g\cap (R\,R^{-1})$, contradicting
the condition of Definition~$\ref{def:2scprop}(3)$. 
Hence $\varphi\in D(G)$. 
}
\end{proof}

It follows from Lemma~\ref{lem:2scsimple}(b) that condition 
$(\hat 2)$ above is equivalent to Definition~$\ref{def:2scprop}(2)$. 
Also the proof above shows that Definition~$\ref{def:2scprop}(3)$
implies condition~$(\hat{3})$, while parts (1) and (3) of
Definition~$\ref{def:2scprop}$ are used together to obtain condition~$(\hat{1})$.

\begin{question}
Under what conditions on $L$ and $R$ do the conditions 
$(\hat 1)$--$(\hat 3)$ imply that $(L,R)$ has 
the $2$S-Cayley property? 
\end{question}

Our next task is to prove Theorem~\ref{thm:conn}.

\medskip\noindent
\textbf{Proof of Theorem~\ref{thm:conn}.}\quad We consider  $\Gamma=\2sc(G;L,R)$, where 
$L, R$ are inverse-closed subsets of a group $G$ with the $2$S-Cayley property.
In particular, by Theorem~\ref{thm:2scsimple}, $\Ga$ is a simple 
undirected graph.

Suppose first that $\Gamma$ is connected. Then, for each $g\in G$, 
there is a path from the vertex $e$ to $g$, and hence $\ell_k^{-1}
\ell_{k-1}^{-1}\ldots \ell_1^{-1}er_1\ldots r_{k-1}r_{k}=g$,
for some $k$ and some $\ell_i\in L, r_i\in R$. Thus $g=xy$ with 
$x=\ell_k^{-1}
\ell_{k-1}^{-1}\ldots \ell_1^{-1}\in \langle L\rangle$
and $y=r_1\dots r_k\in\langle R\rangle$, and so $G=\langle L\rangle\,\langle R\rangle$. 
Choosing the element $g$ to be $\ell\in L$, we obtain an expression
$e=\ell^{-1}xy$ and here $\ell^{-1}x$ is a word in $L=L^{-1}$ of length $k+1$ and
$y$ is a word in $R$ of length $k$. Thus condition $(*)$ holds.

Next suppose that $G=\langle L\rangle\,\langle R\rangle$, and
let $g\in G$. Then we have $g=x_gy_g$, where $x_g, y_g$ are 
words in $L, R$, respectively. 
Suppose first that $(*)$ holds. Then there are words $x_e, y_e$ in $L$ and $R$ 
respectively, with lengths of opposite parity such that
$x_ey_e=e$. If necessary, by  replacing $x_g$ by $x_gx_e$ and $y_g$ by $y_ey_g$, 
we may assume that the lengths of $x_g$ and $y_g$ have the same parity.  
Moreover, by Definition~$\ref{def:2scprop}(3)$, $e\in (L\,L^{_1})\cap(R\,R^{-1})=L^2\cap R^2$, 
so we may modify $x_g$ and $y_g$ further and assume that $x_g$ and $y_g$ 
have the same length. 
These words then give a path in $\Gamma$ from $e$ to $g$. Hence $\Gamma$ is connected.

Now assume that $G=\langle L\rangle\,\langle R\rangle$ with $L, R$ inverse-closed, but that 
condition $(*)$ fails. It follows from our proof above that 
$\Ga$ is not connected. 
Let $g\in G$ be expressed as $g=x_gy_g$ as above. Using the property
 $e\in L^2\cap R^2$, we may modify $x_g, y_g$ so that either (i) $|x_g|=|y_g|$,
or (ii) $|x_g|=1+|y_g|$. Since $(*)$ fails, for each expression $e=x_ey_e$
with $x_e,y_e$ words in $L, R$ respectively, the lengths of $x_e, y_e$ 
must have the same parity. This implies that, for a given element $g$, either
each expression satisfies (i) or each expression satisfies (ii).
Note that (ii) holds for each $\ell\in L$ and each $r\in R$, 
and we have an edge from $\ell$ to $\ell^{-1}\ell r=r$.
If (i) holds for $g$ then we obtain a path in $\Ga$ from $e$ to $g$.
Suppose then that (ii) holds for $g$, and let $\ell$ be the 
last letter of $L$ in the word $x_g$. Then the expression
$g=x_gy_g$ satisfying (ii) gives a path in $\Ga$ from $\ell$ to $g$.
Thus $\Ga$ has two connected components, and they are the sets described in Remark~\ref{rem:conn}.
\hfill{$\blacksquare$}\medskip

\section{Some isomorphisms of two-sided Cayley graphs}\label{sec:iso}

In this section we prove Theorems~\ref{thm:iso} and~\ref{cor:primval}.

\medskip\noindent
\textbf{Proof of Theorem~\ref{thm:iso}.}\quad We consider  
$\Gamma=\2sc(G;L,R)$, where $L, R$ are non-empty subsets of 
a group $G$, and we are given $x,y\in G$, and $\sigma\in \Aut(G)$.
It is straightfoward to check that each of the conditions
of Definition~$\ref{def:2scprop}$ holds for one of the pairs
$(L,R), (R,L), (L^{\sigma},R^{\sigma})$, $(L^{x},R^{y})$
if and only if it holds for all of them, noting, for example,
that $(L^\sigma(L^\sigma)^{-1})^g\cap R^\sigma(R^\sigma)^{-1} = 
((L\,L^{-1})^{\sigma g\sigma^{-1}}\cap R\,R^{-1})^{\sigma}$.
This proves part (a).

Now suppose that $(L,R)$ has the $2$S-Cayley property. 
Then, for each of the graphs $\Ga'$ in Table~\ref{tbl:iso},
it is straightforward to check that the map $\varphi:G\rightarrow G$ 
in the same line of the table is an isomorphism from $\Ga$ to $\Ga'$. 
For example, in the third line, $\varphi$ maps the arc $(g,\ell^{-1}gr)$ 
of $\Ga$ to the pair $(x^{-1}gy, (\ell^x)^{-1}x^{-1}gyr^y)$ which is an arc of
$\2sc(G;L^x,R^y)$.
\hfill{$\blacksquare$}\medskip

\begin{table}
\begin{center}
\begin{tabular}{cl}
\hline
$\Ga'$ & $\varphi$\\ \hline
$\2sc(G;R,L)$ & $g\mapsto g^{-1}$\\
$\2sc(G;L^{\sigma},R^{\sigma})$ & $g\mapsto g^\sigma$\\
$\2sc(G;L^x,R^y)$ & $g\mapsto x^{-1}gy$\\ \hline                                                       
\end{tabular}
\caption{Isomorphisms in the proof of Theorem~\ref{thm:iso}}\label{tbl:iso}
\end{center}
\end{table}

Now we prove Theorem~\ref{cor:primval}, which implies that
the Petersen graph is not a two-sided Cayley graph for any group
of order 10. 

\medskip\noindent
\textbf{Proof of Theorem~\ref{cor:primval}.}\quad 
Let $\Gamma=\2sc(G;L,R)$, where $L, R$ are subsets 
of a group $G$ with the $2$S-Cayley property, 
and suppose that $\Ga$ is connected and regular of 
prime valency $p$,  
that is, each vertex lies on exactly $p$ edges.
Suppose moreover that $\Ga$ is not a Cayley graph. Then
$p$ is an odd prime, since each connected graph of
valency $2$ is a cycle, and hence a Cayley graph.  
By  Theorem~\ref{thm:2scsimple}, the valency of $\Ga$ satisfies 
$p=|L|.|R|$, and so $\{|L|,|R|\}=\{p,1\}$. 
Also, by Theorem~\ref{thm:iso}(b), $\Ga\cong\2sc(G;R,L)$,
and so, up to isomorphism, we may assume that
$|L|=p$ and $|R|=1$. Let $R=\{r\}$. 
If $r\in Z(G)$, 
then each element $g$ is adjacent precisely to the 
vertices $\ell^{-1} gr=\ell^{-1} rg$, for 
$\ell\in L$, and hence $\Ga=\Cay(G,L^{-1}R)$, a contradiction. 
Hence $r$ lies in a $G$-conjugacy class $\mathcal{C}$ of non-central 
elements, and (a) holds. Note that $L\cap\mathcal{C}=\varnothing$ by 
Definition~\ref{def:2scprop}~(2).

Since $\Ga$ is connected, each $g\in G$ can be reached by a path 
from $e$, and by the definition of adjacency and using 
Definition~\ref{def:2scprop}~(1), we find that 
$g=\ell_k^{-1}\dots\ell_1^{-1}er_1\dots r_k$, for some $\ell_i\in L,  r_i\in R$.
Hence  $G=H\,\langle r\rangle$, where $H=\langle L\rangle$.
Now by Definition~\ref{def:2scprop}~(1) with $g=e$,
we have $L^{-1}r=Lr^{-1}$, and hence $r^2\in H$. 
Suppose that $H\ne G$. Then $H$ has index $2$ in $G$. Consider the bijection
$\theta:G\rightarrow H\times \mathbb{Z}_2$ given by 
$\theta: h\mapsto (h,0), hr\mapsto (h,1)$ for $h\in H$.
Then $\theta$ induces an isomorphism from $\Ga$ to $\Sigma:=\Cay(H\times
\mathbb{Z}_2, L')$ where $L'=\{(\ell^{-1},1) | \ell\in L\}$
(since, for all $h\in H$, an arc $(h,\ell^{-1} hr)$ is mapped to an arc 
$((h,0), (\ell^{-1} h,1))$ of $\Sigma$, while an arc $(hr,\ell^{-1} hr^2)$ of $\Ga$  is 
mapped to an arc $((h,1), (\ell^{-1} hr^2,0))$ of $\Sigma$).
This contradiction shows that $H=G$, so part (b) is proved.

If $G$ is abelian then there are no non-central elements, 
so (a) fails and therefore each $\Ga$ is a Cayley graph. Suppose
now that $G=\langle a,b\mid a^n=b^2=e, bab=a^{-1}\rangle\cong
D_{2n}$ with $n$ odd. Let $Z=\langle a\rangle\cong Z_n$, and
note that the set of involutions in $G$ forms a single conjugacy 
class and equals $G\setminus Z$. Let $\Ga, L, H, r, \mathcal{C}$ be as above.
Since $G=H$, it follows that $L$ must contain some element of
$G\setminus Z$, and by Theoprem~\ref{thm:iso}
we may assume that $b\in L$. Since
$L\cap\mathcal{C}=\varnothing$, the element $r$ belongs to $Z$, 
and in particular $|r|$ is odd (dividing $n$, recall that $r\ne e$). 
By  Definition~\ref{def:2scprop}(1), $L^{-1}gr=Lgr^{-1}$, and hence 
$L^{-1}(gr^2g^{-1})=L$, for each $g\in G$. Now $gr^2g^{-1}=r^2$ or $r^{-2}$ 
according as $g\in Z$ or $g\not\in Z$. Thus $L^{-1}r^{\pm 2}=L$. For 
each involution $b'\in L$ (for example, $b'=b$), this implies that
$L$ contains $b's$ for each $s\in\langle r^2\rangle = \langle r\rangle$ 
(recall that $|r|$ is odd). Simlarly if $a^i\in L$ then $a^{-i}r^{\pm2}\in L$,
and hence also $a^ir^{\pm 4}\in L$: so in this case we again find that $L$ 
contains $a^is$ for each   $s\in\langle r^4\rangle = \langle r\rangle$. 
Thus $L$ is a union of cosets of the subgroup $\langle r\rangle$,
and hence $|L|=p$ is divisible by $|r|$. It follows that $|r|=p$,
$L=b\langle r\rangle$, and hence that $G=\langle L\rangle=D_{2p}$
so $n=p$. An easy calculations shows that $\Ga$ is the 
%complete bipartite graph $K_{p,p}$ and is the 
Cayley graph $\Cay(G,L)$.
This contradiction completes the proof. 
\hfill{$\blacksquare$}\medskip

\section{Two-sided Cayley graphs and Cayley graphs}\label{sec:cay}

% \begin{theorem}\label{twosidedisCayley}
% Let  $L, R$ be subsets of a group $G$ with the $2$S-Cayley property, and let $\Gamma=\2sc(G;L,R)$.  Then the following are equivalent. 
% \begin{enumerate}
% \item[(a)]  $\Gamma$ is a Cayley graph for $G$;
% \item[(b)]  for each $g\in G$, $LgR=LRg$;
% \item[(c)]  for each $g\in G$, $LgR=gLR$.
% \end{enumerate}
% \end{theorem}

We begin by proving Theorem \ref{twosidedisCayley}, which gives  
sufficient conditions on $L, R$ for a two-sided Cayley graph $\2sc(G;L,R)$ to 
be a Cayley graph. Recall, for a group $G$, the regular permutation groups 
in $\Sym(G)$, $G_R=\{\lambda_{e,g}|g\in G\}$ and $G_L=\{\lambda_{g,e}|g\in G\}$,
each isomorphic to $G$.

\medskip\noindent
\textbf{Proof of Theorem~\ref{twosidedisCayley}.}\quad 
Let  $L, R$ be subsets of a group $G$ with the $2$S-Cayley property, and 
let $\Gamma=\2sc(G;L,R)$. By Definition~\ref{def:2scprop}(1), 
each arc of $\Ga$ has the form
$a=(x,\ell^{-1} xr)$ for some $x\in G, \ell\in L, r\in R$.

Suppose first that $L^{-1}gR=L^{-1}Rg$ for each $g\in G$. Then 
$a=(x,\ell^{-1} r'x)$ for some $r'\in R$.
The image of the arc $a$ under an element $\lambda_{e,g}
\in G_R$ is the pair $(xg,\ell^{-1} r'xg)=(xg,\ell^{-1} xgr'')$, for some $r''\in R$ (using
$L^{-1}(xg)R=L^{-1}R(xg)$), and this is again an arc. Hence $\lambda_{e,g}$ is an automorphism of 
$\Ga$, and it follows that $\Aut(\Ga)$ contains the regular subgroup $G_R$. Thus,
by Sabidussi's Theorem \cite{MR0097068}, $\Ga$ is a Caley graph for $G$. Indeed,
since each vertex $x$ is adjacent precisely to the vertices in $L^{-1}Rx$ it follows that 
$\Ga=\Cay(G,L^{-1}R)$. Conversely if $G_R\leq\Aut(\Ga)$, then the set of vertices
adjacent to $x=(e)\lambda_{e,x}$ is equal to the image under $\lambda_{e,x}$
of the set $L^{-1}R$ of vertices adjacent to $e$, and this is the set $L^{-1}Rx$.
Thus $L^{-1}xR=L^{-1}Rx$ for each $x$. Here $\Ga= \Cay(G, L^{-1}R)$.

A similar proof shows that, the condition $L^{-1}gR=gL^{-1}R$,
for each $g\in G$, holds if and only if $\Aut(\Ga)$ contains the regular 
subgroup $G_L$. (In this case $\Ga\cong \Cay(G,R^{-1}L)$ under the isomorphism 
$g\mapsto g^{-1}$ for $g\in G$.)
%
% Now suppose that $\Gamma$ is a Cayley graph and so $\hat{G}$ is a regular 
% subgroup of $Aut(\Gamma)$. Then for each $\hat{g}\in \hat{G}$ we have $(N_{\Gamma}(e))\hat{g}=N_{\Gamma}((e)\hat{g})$. Therefore, $L^{-1}eRg=L^{-1}egR$ or $LRg=LgR$.
\hfill{$\blacksquare$}\medskip

Now we give sufficient conditions for $2SCay(G;L,R)$ to be a Cayley graph, 
as mentioned in Remark~\ref{rem:isCayley}. 
%Also we obtain a sufficient condition for $2SCay(G;L,R)$ not to be a Cayley graph. 
For any subset $H$ of a group $G$ , the \textit{normalizer} of $H$ in $G$ 
is $N_G(H)=\{g\in G\, |\, Hg=gH\}$. It is a subgroup of $G$. 

\begin{proposition}\label{twosidedisCayley1}
{Let  $L, R$ be subsets of a group $G$ with the $2$S-Cayley property, and let $\Gamma=\2sc(G;L,R)$.
Then $\Gamma$ is a Cayley graph if at least one of the following conditions holds :
\[
(\delta)\ (N_G(L^{-1})\cap N_G(L^{-1}R)).N_G(R)=G,
\]
\[
(\delta')\ N_G(L^{-1}).(N_G(L^{-1}R)\cap N_G(R))=G.
\]
In particular, if $G=N_G(L)$ then $\Ga=\Cay(G,L^{-1}R)$, and if $G=N_G(R)$
then $\Ga\cong\Cay(G,R^{-1}L)$.
}\end{proposition}

\begin{proof}
{Let $g\in G$. We show that $L^{-1}gR=L^{-1}Rg$ if condition ($\delta$) holds. From $(\delta)$ 
we have $g=xy$ where $x\in N_G(L^{-1})\cap N_G(L^{-1}R))$ and $y\in N_G(R)$. 
Therefore, $L^{-1}gR=L^{-1}xyR=xL^{-1}Ry=L^{-1}Rxy=L^{-1}Rg$ and then, by Theorem~\ref{twosidedisCayley},
$\Ga=\Cay(G,L^{-1}R)$.
A similar argument shows that condition $(\delta')$ 
implies that $L^{-1}gR=gL^{-1}R$ for each $g$ in $G$. 
Then by Theorem~\ref{twosidedisCayley},
$\Ga\cong \Cay(G,R^{-1}L)$.
}\end{proof}

Finally, as mentioned in Remark~\ref{rem:isCayley}, we study 
elements in $\langle G_L, G_R\rangle = \{\lambda_{x,y}\,|\, x, y\in G\}$
which act as automorphisms of $\2sc(G:L,R)$, and obtain a sufficient condition for vert-transitivity.

\begin{proposition}\label{vt}
Let  $L, R$ be non-empty subsets of a group $G$, and let $\Gamma=\2sc(G;L,R)$.
Then $\Aut(\Ga)$ contains $K:=\{\lambda_{x,y}\,|\, x\in N_G(L), y\in N_G(R)\}$
and $N_{\Aut(G)}(L)\cap N_{\Aut(G)}(R)$. In particular, $K$ acts vertex-transitively 
on $\Ga$ if and only if $G=N_G(L)N_G(R)$.
\end{proposition}

\begin{proof}{
We note that $N_G(L)$ is a subgroup and is equal to $N_G(L^{-1})$, 
with similar comments for $N_G(R)$. 
Let $\lambda_{x,y}\in K$, and let $a=(g,\ell^{-1}gr))$ be an arc of $\Ga$.
The image of $a$ under $\lambda_{x,y}$ is the pair $(x^{-1}gy,x^{-1}\ell^{-1}gry)$,
and this is an arc of $\Ga$ since $x^{-1}\ell^{-1}gry=(\ell')^{-1}(x^{-1}gy)r'$, for 
some $\ell'\in L$ and $r'\in R$, by the definition of $K$.
Let $\sigma\in N_{\Aut(G)}(L)\cap N_{\Aut(G)}(R)$. Then the image of $a$
under $\sigma$ is $(g^\sigma, \ell^{-\sigma}g^\sigma r^\sigma)$, and 
this is an arc of $\Ga$ since $\ell^\sigma\in L, r^\sigma\in R$.

Let $g\in G$. Suppose first that $G=N_G(L)N_G(R)$.  Then $g=xy$ for 
some $x\in N_G(L), y\in N_G(R)$, and  $\lambda_{x,y}\in K$
maps $e$ to $g$ so $K$ is vertex-transitive. Conversely suppose that 
$K$ is vertex-transitive, so some $\lambda_{x,y}\in K$ maps $e$ to $g$.
Then $g=x^{-1}ey\in N_G(L)N_G(R)$.
}
\end{proof}

The conditions involved in the $2S$-Cayley property become much simpler 
in the case where $G=N_G(L)N_G(R)$.

\begin{proposition}\label{simple}
 Let  $L, R$ be non-empty subsets of a group $G$ such that 
$G=N_G(L)N_G(R)$, and let $\Gamma=\2sc(G;L,R)$. Then $(L,R)$ 
has the $2S$-Cayley property if and only if the following all hold.
\begin{enumerate}
 \item[$(1')$] $L^{-1}R=LR^{-1}$;
 \item[$(2')$] $L\cap R=\varnothing$;
 \item[$(3')$] $L^{-1}L\cap R^{-1}R=\{e\}$. 
\end{enumerate}
\end{proposition}

\begin{proof}
{We give full details for part $(2')$.
Let $g\in G$. Since $N_G(L)N_G(R)=G$, $g=\ell r$ for some $\ell\in N_G(L)$ 
and $r\in N_G(R)$. Recall that $N_G(L)=N_G(L^{-1})$, etc. 
\begin{center}
$\begin{array}{lll}
  L^g\cap R=\varnothing & \Longleftrightarrow & L^{\ell r}\cap R=\varnothing \\
   & \Longleftrightarrow & L^{\ell}\cap R^{r^{-1}}=\varnothing \\
   & \Longleftrightarrow & L\cap R=\varnothing.
\end{array}$
\end{center}
Thus  Definition~\ref{def:2scprop}(2) holds if and only if $L\cap R=\varnothing$. 
Similar proofs show that, for $i=1,3$, Definition~\ref{def:2scprop}$(i)$ holds  
if and only if $(i')$ holds.
}\end{proof}

{\bf Acknowledgements}

The first author acknowledges support of an Endeavour Award of the 
Australian Government for his study leave during which the research for this
paper began. He thanks the School of Mathematics and Statistics, 
University of Western Australia, for their warm hospitality during his visiting 
appointment and for the facilities and help provided. In addition he thanks 
the Iranian Elite National Foundation (Bonyad Melli Nokhbegan) for financial 
support and the Department of Mathematical Sciences, Sharif University of Technology, 
for their hospitality during his postdoctoral appointment and for the facilities.
%We would also like to thank the referees for their valuable suggestions.

% -----------------------------------------------------------
% \bibliographystyle{amsplain}
% \bibliography{xbib}
% \end{document}
% -----------------------------------------------------------

% \bibitem{ser98a}
% [23] \'{A}kos Seress, On vertex-transitive, non-Cayley graphs of order $pqr$, \emph{Discrete Math.} 182
% (1998), no. 1--3, 279--292, Graph theory (Lake Bled, 1995). MR MR1603711 (98i:05095)

\end{document}